\documentclass[12pt]{article}
\usepackage[a4paper,margin=1.1in]{geometry}
\usepackage{amssymb}
\usepackage{amsmath}
\usepackage{amsfonts}
\usepackage{amsthm}
\usepackage{color}
\usepackage{float}
\usepackage[all]{xy}
\usepackage{scalerel}
\usepackage{mathabx}
\usepackage{delimset}
\usepackage{setspace}

\def\H{\mathop{\mbox{\textup{H}}}\nolimits}

\newcommand{\qbinom}{\genfrac{[}{]}{0pt}{}}

\newcommand{\C}{\mathbb{C}}

\newcommand{\N}{\mathbb{N}}

\newcommand{\e}{\mathrm{e}}

\newcommand{\E}{\mathrm{E}}
\newcommand{\T}{\mathrm{T}}
\newcommand{\rr}{\mathrm{r}}

\newtheorem{theorem}{Theorem}

\newtheorem{definition}{Definition}

\begin{document}

\title{\textbf{On the Heine Binomial Operators}}
\author{Ronald Orozco L\'opez}
\newcommand{\Addresses}{{% additional braces for segregating \footnotesize
  \bigskip
  \footnotesize

  %R.~Orozco, \textsc{Departamento de Matematicas, Universidad de los Andes, 
  % Carrera 1 N. 18A-12 Bogot\'a, Colombia}\par\nopagebreak
  \textit{E-mail address}, R.~Orozco: \texttt{rj.orozco@uniandes.edu.co}
  
}}

\maketitle
%\tableofcontents

\begin{abstract}
In this paper, we introduce the Heine binomial operators $\H_{n}(bD_{q})$ based on $q$-differential operator $D_{q}$. The motivation for introducing the operators $\H_{n}(bD_{q})$ is that their limit turns out to be the $q$-exponential operator $\T(bD_{q})$ given by Chen. The Hahn polynomials $\Phi_{m}^{(q^n)}(b,x|q)$ can easily be represented by using the operators $\H_{n}(bD_{q})$. Here, we derive $q$-exponential and ordinary generating function, Mehler's formula, Rogers formula, and other identities for the polynomials $\Phi_{m}^{(q^n)}(b,x|q)$.
\end{abstract}
\noindent 2020 {\it Mathematics Subject Classification}:
Primary 05A30. Secondary 33D45.

\noindent \emph{Keywords: } Hahn polynomials, $(s,t)$-exponential operator, Heine binomial operators, Rogers-type formula, Mehler-type formula.

\section{Introduction}

We begin with some notation and terminology for basic hypergeometric series \cite{gasper}. Let $\vert q\vert<1$ and the $q$-shifted factorial be defined by
\begin{align*}
    (a;q)_{n}&=\prod_{k=0}^{n-1}(1-q^{k}a),\\
    (a;q)_{\infty}&=\lim_{n\rightarrow\infty}(a;q)_{n}=\prod_{k=0}^{\infty}(1-aq^{k}).
\end{align*}
The multiple $q$-shifted factorial is defined by
\begin{equation*}
    (a_{1},a_{2},\ldots,a_{m};q)_{\infty}=(a_{1};q)_{\infty}(a_{2};q)_{\infty}\cdots(a_{m};q)_{\infty}.
\end{equation*}
The $q$-binomial coefficient is defined by
\begin{equation*}
\qbinom{n}{k}_{q}=\frac{(q;q)_{n}}{(q;q)_{k}(q;q)_{n-k}}.
\end{equation*}
The ${}_{r+1}\phi_{r}$ basic hypergeometric series is define by
\begin{equation*}
    {}_{r+1}\phi_{r}\left(
    \begin{array}{c}
         a_{1},a_{2},\ldots,a_{r+1} \\
         b_{1},\ldots,b_{r}
    \end{array}
    ;q,z
    \right)=\sum_{n=0}^{\infty}\frac{(a_{1},a_{2},\ldots,a_{r};q)_{n}}{(q;q)_{n}(b_{1},b_{2},\ldots,b_{r};q)_{n}}z^n,\ \vert z\vert<1.
\end{equation*}
In this paper, we will frequently use the $q$-binomial theorem:
\begin{equation}
    {}_1\phi_{0}(a;q,z)=\frac{(az;q)_{\infty}}{(z;q)_{\infty}}=\sum_{n=0}^{\infty}\frac{(a;q)_{n}}{(q;q)_{n}}z^{n}.
\end{equation}
The $q$-exponential $\e_{q}(z)$ is defined by
\begin{equation*}
    \e_{q}(z)=\sum_{n=0}^{\infty}\frac{z^n}{(q;q)_{n}}={}_1\phi_{0}\left(\begin{array}{c}
         0\\
         - 
    \end{array};q,-z\right)=\frac{1}{(z;q)_{\infty}}.
\end{equation*}
Another $q$-analogue of the classical exponential function is
\begin{equation*}
    \E_{q}(z)=\sum_{n=0}^{\infty}q^{\binom{n}{2}}\frac{z^n}{(q;q)_{n})}={}_1\phi_{1}\left(\begin{array}{c}
         0\\
         0 
    \end{array};q,-z\right)=(-z;q)_{\infty}.
\end{equation*}
The following easily verified identities will be frequently used in this paper:
\begin{align}
    (a;q)_{n}&=\frac{(a;q)_{\infty}}{(aq^n;q)_{\infty}},\label{eqn_iden1}\\
    (a;q)_{n+k}&=(a;q)_{n}(aq^n;q)_{k},\label{eqn_iden2}\\
    (a;q)_{n-k}&=\frac{(a;q)_{n}}{(a^{q}q^{1-n};q)_{k}}(-qa^{-1})^{k}q^{\binom{k}{2}-nk}.\label{eqn_iden3}
\end{align}
The $q$-differential operator $D_{q}$ is defined by:
\begin{equation*}
    D_{q}f(x)=\frac{f(x)-f(qx)}{x}
\end{equation*}
and the Leibniz rule for $D_{q}$
\begin{equation}\label{eqn_leibniz}
    D_{q}^{n}\{f(x)g(x)\}=\sum_{k=0}^{n}q^{k(n-k)}\qbinom{n}{k}_{q}D_{q}^{k}\{f(x)\}D_{q}^{n-k}\{g(q^{k}x)\}.
\end{equation}
Chen and Liu \cite{chen1} introduced the $q$-exponential operator
\begin{equation*}
    \T(bD_{q})=\sum_{n=0}^{\infty}\frac{(bD_{q})^n}{(q;q)_{n}}.
\end{equation*}
In this paper, we need the homogeneous Hahn polynomials, which are defined in \cite{liu} by
\begin{equation*}
    \Phi_{n}^{(\alpha)}(x,y|q)=\sum_{k=0}^{n}\qbinom{n}{k}_{q}(\alpha;q)_{k}x^ky^{n-k}.
\end{equation*}

\section{Heine binomial operators}

The Heine binomial formula is
\begin{equation}
    \frac{1}{(x;q)_{n}}=\sum_{k=0}^{\infty}\qbinom{n+k-1}{k}_{q}x^k.
\end{equation}
The function $q$-exponential $\e_{q}(x)$ is the limit of the Heine binomial formula,
\begin{equation*}
    \e_{q}(x)=\sum_{n=0}^{\infty}\frac{x^n}{(q;q)_{n}}=\frac{1}{(x;q)_{\infty}}=\lim_{n\rightarrow\infty}\sum_{k=0}^{\infty}\qbinom{n+k-1}{k}_{q}x^k.
\end{equation*}

\begin{definition}
We define the $n$-th Heine binomial operator based on $D_{q}$ by
\begin{equation}
    \H_{n}(bD_{q})=\sum_{k=0}^{\infty}\qbinom{n+k-1}{k}_{q}(bD_{q})^k
\end{equation}
\end{definition}
$\H_{n}(bD_{q})$ can be written as
\begin{equation}\label{eqn_heine}
    \H_{n}(bD_{q})=\sum_{k=0}^{\infty}\frac{(q^n;q)_{k}}{(q,q)_{k}}(bD_{q})^k.
\end{equation}
Using the Heine formula, $\H_{n}(bD_{q})$ can also written as $1/(bD_{q};q)_{n}$. 
The $q$-exponential operator $\T(bD_{q})$ is the limit of the Heine binomial operators $\H_{n}(bD_{q})$,
\begin{equation*}
    \T(bD_{q})=\lim_{n\rightarrow\infty}\H_{n}(bD_{q}).
\end{equation*}
\begin{theorem}
\begin{align*}
    \H_{n}(bD_{q})\{x^m\}=
    \begin{cases}
        x^m,&\text{ if }n=0;\\
        \Phi_{m}^{(q^n)}(b,x|q)=\sum_{k=0}^{m}\qbinom{m}{k}_{q}(q^n;q)_{k}b^kx^{m-k},&\text{ if }n\geq1;\\
        \rr_{m}(b,x|q),&\text{ if }n=\infty.
    \end{cases}
\end{align*}
\end{theorem}
\begin{proof}
$\H_{0}(bD_{q})=1$ is the identity operator. Then $\H_{0}(bD_{q})\{x^m\}=x^m$. For $n\neq1$
    \begin{align*}
        \H_{n}(bD_{q})\{x^m\}&=\sum_{k=0}^{\infty}\frac{(q^n;q)_{k}}{(q;q)_{k}}b^kD_{q}^{k}x^m\\
        &=\sum_{k=0}^{m}\frac{(q^n;q)_{k}}{(q;q)_{k}}\frac{(q;q)_{m}}{(q;q)_{m-k}}b^kx^{m-k}\\
        &=\sum_{k=0}^{m}\qbinom{m}{k}_{q}(q^n;q)_{k}b^kx^{m-k}\\
        &=\Phi_{m}^{(q^n)}(b,x|q).
    \end{align*}
Finally,
\begin{equation*}
    \lim_{n\rightarrow\infty}\H_{n}(bD_{q})\{x^m\}=\sum_{k=0}^{m}\qbinom{m}{k}_{q}b^kx^{m-k}=\rr_{m}(x,b|q).
\end{equation*}
The proof is completed.
\end{proof}

\begin{theorem}
For all $n\in\N$
    \begin{equation*}
        \H_{n}(bD_{q})\left\{\frac{1}{(ax;q)_{\infty}}\right\}=\frac{(q^nab;q)_{\infty}}{(ax,ab;q)_{\infty}}.
    \end{equation*}
\end{theorem}
\begin{proof}
    \begin{align*}
        \H_{n}(bD_{q})\left\{\frac{1}{(ax;q)_{\infty}}\right\}&=\sum_{k=0}^{\infty}\frac{(q^n;q)_{k}}{(q;q)_{k}}b^kD_{q}^{k}\left\{\frac{1}{(ax;q)_{\infty}}\right\}\\
        &=\sum_{k=0}^{\infty}\frac{(q^n,q)_{k}}{(q;q)_{k}}\frac{b^ka^k}{(ax;q)_{\infty}}\\
        &=\frac{1}{(ax;q)_{\infty}}\sum_{k=0}^{\infty}\frac{(q^n;q)_{k}}{(q;q)_{k}}(ab)^k\\
        &=\frac{1}{(ax;q)_{\infty}}\frac{(q^nab,q)_{\infty}}{(ab;q)_{\infty}}\\
        &=\frac{(q^nab;q)_{\infty}}{(ax,ab;q)_{\infty}}
    \end{align*}
as claimed.    
\end{proof}

\begin{equation*}
    \T(bD_{q})\left\{\frac{1}{(ax;q)_{\infty}}\right\}=\lim_{n\rightarrow\infty}\frac{(q^nab;q)_{\infty}}{(ax,ab;q)_{\infty}}=\frac{1}{(ax,ab;q)_{\infty}}
\end{equation*}

\begin{theorem}
For all $n\in\N$
    \begin{equation*}
        \H_{n}(bD_{q})\{(ax;q)_{\infty}\}=(ax;q)_{\infty}{}_{2}\varphi_{1}\left(
        \begin{array}{c}
             q^n,0 \\
             -ax
        \end{array};
        q,ab
        \right).
    \end{equation*}
\end{theorem}
\begin{proof}
    \begin{align*}
        \H_{n}(bD_{q})\{(ax;q)_{\infty}\}&=\sum_{k=0}^{\infty}\frac{(q^n;q)_{\infty}}{(q;q)_{\infty}}b^kD_{q}^{k}\{(ax;q)_{\infty}\}\\
        &=\sum_{k=0}^{\infty}\frac{(q^n;q)_{k}}{(q;q)_{k}}b^ka^k(axq^k;q)_{\infty}\\
        &=(ax,q)_{\infty}\sum_{k=0}^{\infty}\frac{(q^n;q)_{k}}{(q;q)_{k}(ax;q)_{k}}(ab)^k\\
        &=(ax;q)_{\infty}{}_{2}\varphi_{1}\left(
        \begin{array}{c}
             q^n,0 \\
             ax
        \end{array};
        q, a b
        \right).
    \end{align*}
\end{proof}

\begin{equation*}
    \T(bD_{q})\{(ax;q)_{\infty}\}=(ax;q)_{\infty}{}_{2}\varphi_{1}\left(
        \begin{array}{c}
             0,0 \\
             ax
        \end{array};
        q, a b
        \right)
\end{equation*}

\begin{theorem}
    \begin{equation*}
        \H_{n}(aD_{q})\left\{\frac{1}{(bx,cx;q)_{\infty}}\right\}=\frac{(q^{n}ac;q)_{\infty}}{(bx,cx,ac;q)_{\infty}}{}_{2}\varphi_{1}\left(
        \begin{array}{c}
             q^{n},cx  \\
             q^{n}ac
        \end{array};
        q,ab
        \right).
    \end{equation*}
\end{theorem}
\begin{proof}
From Eq.(\ref{eqn_leibniz}), we get
\begin{align*}
    &\H_{n}(aD_{q})\left\{\frac{1}{(bx,cx;q)_{\infty}}\right\}\\
        &\hspace{1cm}=\sum_{k=0}^{\infty}\frac{(q^n;q)_{k}}{(q;q)_{k}}a^{k}D_{q}^{k}\left\{\frac{1}{(bx,cx;q)_{\infty}}\right\}\\
        &\hspace{1cm}=\sum_{k=0}^{\infty}\frac{(q^n;q)_{k}}{(q;q)_{k}}a^{k}\sum_{m=0}^{k}q^{m(m-k)}\qbinom{k}{m}_{q}D_{q}^{m}\left\{\frac{1}{(bx;q)_{\infty}}\right\}D_{q}^{k-m}\left\{\frac{1}{(cxq^{m};q)_{\infty}}\right\}\\
        &\hspace{1cm}=\sum_{k=0}^{\infty}\frac{(q^n;q)_{k}}{(q;q)_{k}}a^{k}\sum_{m=0}^{k}\qbinom{k}{m}_{q}\frac{b^m}{(bx;q)_{\infty}}\frac{c^{k-m}}{(cxq^{m};q)_{\infty}}.
\end{align*}
Reordering sums and from Eq.(\ref{eqn_iden1}), we get
    \begin{align*}
        &\H_{n}(aD_{q})\left\{\frac{1}{(bx,cx;q)_{\infty}}\right\}\\
        &\hspace{1cm}=\frac{1}{(bx,cx;q)_{\infty}}\sum_{m=0}^{\infty}\frac{b^m(cx;q)_{m}}{(q;q)_{m}}\sum_{k=m}^{\infty}\frac{(q^n;q)_{k}a^kc^{k-m}}{(q;q)_{k-m}}\frac{}{}\\
        &\hspace{1cm}=\frac{1}{(bx,cx;q)_{\infty}}\sum_{m=0}^{\infty}\frac{(ab)^m(cx;q)_{m}}{(q;q)_{m}}\sum_{k=0}^{\infty}\frac{(q^n;q)_{k+m}(ac)^k}{(q;q)_{k}}\frac{}{}\\
        &\hspace{1cm}=\frac{1}{(bx,cx;q)_{\infty}}\sum_{m=0}^{\infty}\frac{(q^n;q)_{m}(cx;q)_{m}}{(q;q)_{m}}(ab)^m\sum_{k=0}^{\infty}\frac{(q^{n+m};q)_{k}}{(q;q)_{k}}(ac)^k\\
        &\hspace{1cm}=\frac{1}{(bx,cx;q)_{\infty}}\sum_{m=0}^{\infty}\frac{(q^n;q)_{m}(cx;q)_{m}}{(q;q)_{m}}(ab)^m\frac{(q^{n+m}ac;q)_{\infty}}{(ac;q)_{\infty}}\\
        &\hspace{1cm}=\frac{(q^{n}ac;q)_{\infty}}{(bx,cx,ac;q)_{\infty}}\sum_{m=0}^{\infty}\frac{(q^n;q)_{m}(cx;q)_{m}}{(q;q)_{m}}(ab)^m\frac{1}{(q^nac;q)_{m}}\\
        &\hspace{1cm}=\frac{(q^{n}ac;q)_{\infty}}{(bx,cx,ac;q)_{\infty}}{}_{2}\varphi_{1}\left(
        \begin{array}{c}
             q^{n},cx  \\
             q^{n}ac
        \end{array};
        q, ab
        \right).
    \end{align*}
The proof is reached.    
\end{proof}
Reordenando sumas
\begin{align*}
    \T(aD_{q})\left\{\frac{1}{(bx,cx;q)_{\infty}}\right\}&=\lim_{n\rightarrow\infty}\frac{(q^{n}ac;q)_{\infty}}{(bx,cx,ac;q)_{\infty}}{}_{2}\varphi_{1}\left(
        \begin{array}{c}
             q^{n},cx  \\
             q^{n}ac
        \end{array};
        q, ab
        \right)\\
        &=\frac{1}{(bx,cx,ac;q)_{\infty}}{}_{2}\varphi_{1}\left(
        \begin{array}{c}
             0, cx  \\
             0
        \end{array};
        q, ab
        \right)\\
        &=\frac{1}{(bx,cx,ac;q)_{\infty}}\frac{(a b cx;q)_{\infty}}{(ab;q)_{\infty}}\\
        &=\frac{(a b cx;q)_{\infty}}{(bx,cx,ac,ab;q)_{\infty}}.
\end{align*}

\begin{theorem}
    \begin{equation*}
        \H_{n}(dD_{q})\left\{\frac{(cx;q)_{\infty}}{(ax,bx;q)_{\infty}}\right\}=\frac{(cx;q)_{\infty}}{(ax,bx;q)_{\infty}}\sum_{l=0}^{\infty}\frac{(q^n;q)_{l}}{(q;q)_{l}(cx;q)_{l}}b^l\sum_{m=0}^{l}\qbinom{l}{m}_{q}\Phi_{m}^{(bx)}(a,b)c^{l-m}.
    \end{equation*}
\end{theorem}
\begin{proof}
    \begin{align*}
        &\H_{n}(dD_{q})\left\{\frac{(cx;q)_{\infty}}{(ax,bx;q)_{\infty}}\right\}\\
        &\hspace{1cm}=\sum_{l=0}^{\infty}\frac{(q^n;q)_{l}}{(q;q)_{l}}d^lD_{q}^{l}\left\{\frac{(cx;q)_{\infty}}{(ax,bx;q)_{\infty}}\right\}\\
        &\hspace{1cm}=\sum_{l=0}^{\infty}\frac{(q^n;q)_{l}}{(q;q)_{l}}d^l\sum_{m=0}^{l}\qbinom{l}{m}_{q}q^{m(m-l)}D_{q}^{m}\left\{\frac{1}{(ax,bx;q)_{\infty}}\right\}D_{q}^{l-m}\{(q^{m}cx;q)_{\infty}\}
    \end{align*}
As
\begin{align*}
    D_{q}^{m}\left\{\frac{1}{(ax,bx;q)_{\infty}}\right\}&=\sum_{k=0}^{m}q^{k(k-m)}\qbinom{m}{k}_{q}D_{q}^{k}\left\{\frac{1}{(ax;q)_{\infty}}\right\}D_{q}^{m-k}\left\{\frac{1}{(bxq^{k};q)_{\infty}}\right\}\\
    &=\sum_{k=0}^{m}q^{k(k-m)}\qbinom{m}{k}_{q}\frac{a^k}{(ax;q)_{\infty}}\frac{(bq^k)^{m-k}}{(bxq^{k};q)_{\infty}}\\
    &=\sum_{k=0}^{m}\qbinom{m}{k}_{q}\frac{a^k}{(ax;q)_{\infty}}\frac{b^{m-k}}{(bxq^{k};q)_{\infty}}\\
    &=\frac{1}{(ax,bx;q)_{\infty}}\sum_{k=0}^{m}\qbinom{m}{k}_{q}(bx;q)_{k}a^kb^{m-k}\\
    &=\frac{\Phi_{m}^{(bx)}(a,b)}{(ax,bx;q)_{\infty}}.
\end{align*}
Then
\begin{align*}
    &\H_{n}(dD_{q})\left\{\frac{(cx;q)_{\infty}}{(ax,bx;q)_{\infty}}\right\}\\
    &\hspace{1cm}=\frac{1}{(ax,bx;q)_{\infty}}\sum_{l=0}^{\infty}\frac{(q^n;q)_{l}}{(q;q)_{l}}(q^{l}cx;q)_{\infty}d^l\sum_{m=0}^{l}\qbinom{l}{m}_{q}\Phi_{m}^{(bx)}(a,b)c^{l-m}\\
    &\hspace{1cm}=\frac{(cx;q)_{\infty}}{(ax,bx;q)_{\infty}}\sum_{l=0}^{\infty}\frac{(q^n;q)_{l}}{(q;q)_{l}(cx;q)_{l}}d^l\sum_{m=0}^{l}\qbinom{l}{m}_{q}\Phi_{m}^{(bx)}(a,b)c^{l-m}.%\\
    %&\hspace{1cm}=\frac{(cx;q)_{\infty}}{(ax,bx;q)_{\infty}}\sum_{m=0}^{\infty}\frac{\Phi_{m}^{(bx)}(a,b)}{(q;q)_{m}}\sum_{l=m}^{\infty}\frac{(q^n;q)_{l}}{(q;q)_{l-m}(cx;q)_{l}}d^lc^{l-m}\\
    %&\hspace{1cm}=\frac{(cx;q)_{\infty}}{(ax,bx;q)_{\infty}}\sum_{m=0}^{\infty}\frac{(q^n;q)_{m}\Phi_{m}^{(bx)}(a,b)}{(q;q)_{m}(cx;q)_{m}}d^m\sum_{l=0}^{\infty}\frac{(q^{n+m};q)_{l}}{(q;q)_{l}(cxq^m;q)_{l}}(cd)^{l}\\
\end{align*}
\end{proof}

\begin{theorem}
For all $n\in\N$ and $u\in\C$
    \begin{equation*}
        \sum_{k=0}^{\infty}q^{-\binom{k}{2}}\Phi_{k}^{(q^n)}(b,x|q)y^k=\sum_{k=0}^{\infty}(q^n;q)_{k}(by)^kq^{-\binom{k}{2}}(1\ominus_{1,1}qxy)_{q}^{(-k-1)}.
    \end{equation*}
\end{theorem}
\begin{proof}
We have
\begin{align*}
    \sum_{k=0}^{\infty}q^{-\binom{k}{2}}\Phi_{k}^{(q^n)}(b,x|q)y^k
    &=\sum_{k=0}^{\infty}q^{-\binom{k}{2}}\H_{n}(bD_{q})\{x^k\}y^k\\
    &=\H_{n}(bD_{q})\left\{\sum_{k=0}^{\infty}q^{-\binom{k}{2}}(xy)^k\right\}\\
    &=\H_{n}(bD_{q})\{\Theta_{0}(xy,q^{-1})\}\\
    &=\sum_{k=0}^{\infty}\frac{(q^n;q)_{k}}{(q;q)_{k}}b^kD_{q}^k\{\Theta_{0}(xy,q^{-1})\}\\
    &=\sum_{k=0}^{\infty}(q^n;q)_{k}(by)^kq^{-\binom{k}{2}}(1\ominus_{1,1}qxy)_{q}^{(-k-1)}.
\end{align*}
The proof is completed.
\end{proof}

\begin{align*}
    \lim_{n\rightarrow\infty}\sum_{k=0}^{\infty}q^{-\binom{k}{2}}\Phi_{k}^{(q^n)}(b,x|q)y^k&=\lim_{n\rightarrow\infty}\sum_{k=0}^{\infty}(q^n;q)_{k}(by)^kq^{-\binom{k}{2}}(1\ominus_{1,1}qxy)_{q}^{(-k-1)}\\
    &=\sum_{k=0}^{\infty}(by)^kq^{-\binom{k}{2}}(1\ominus_{1,1}qxy)_{q}^{(-k-1)}\\
    &=
\end{align*}

\begin{theorem}
    \begin{equation}
        \sum_{m=0}^{\infty}\Phi_{m}^{(q^n)}(b,x|q)y^m=\frac{1}{1-xy}{}_{2}\varphi_{1}\left(
        \begin{array}{c}
             q^n,q \\
             q xy
        \end{array};
        q, by
        \right).
    \end{equation}
\end{theorem}
\begin{proof}
    \begin{align*}
        \sum_{m=0}^{\infty}\Phi_{m}^{(q^n)}(b,x|q)y^m&=\sum_{m=0}^{\infty}\H_{n}(bD_{q})\{x^m\}y^m\\
        &=\H_{n}(bD_{q})\left\{\sum_{m=0}^{\infty}(xy)^m\right\}\\
        &=\H_{n}(bD_{q})\left\{\frac{1}{1-xy}\right\}\\
        &=\sum_{k=0}^{\infty}\frac{(q^n;q)_{k}}{(q;q)_{k}}b^kD_{q}^k\left\{\frac{1}{1-xy}\right\}\\
        &=\frac{1}{1-xy}\sum_{k=0}^{\infty}\frac{(q^n;q)_{k}(q;q)_{k}}{(qx y;q)_{k}(q;q)_{k}}(by)^k\\
        &=\frac{1}{1-xy}{}_{2}\varphi_{1}\left(
        \begin{array}{c}
             q^n,q \\
             qx y
        \end{array};
        q, by 
        \right).
    \end{align*}
The proof is reached.
\end{proof}

\begin{theorem}
    \begin{equation*}
        \sum_{m=0}^{\infty}\Phi_{m}^{(q^n)}(b,x|q)\frac{y^m}{(q;q)_{m}}=\frac{(q^{n}by;q)_{\infty}}{(xy,by;q)_{\infty}}.
    \end{equation*}
\end{theorem}
\begin{proof}
    \begin{align*}
        \sum_{m=0}^{\infty}\Phi_{m}^{(q^n)}(b,x|q)\frac{y^m}{(q;q)_{m}}&=\sum_{m=0}^{\infty}\H_{n}(bD_{q})\{x^m\}\frac{y^m}{(q;q)_{m}}\\
        &=\H_{n}(bD_{q})\left\{\sum_{m=0}^{\infty}\frac{(xy)^m}{(q;q)_{m}}\right\}\\
        &=\frac{(q^{n}by;q)_{\infty}}{(xy,by;q)_{\infty}}
    \end{align*}
\end{proof}

\begin{theorem}
    \begin{equation*}
        \sum_{m=0}^{\infty}q^{\binom{m}{2}}\Phi_{m}^{(q^n)}(b,x|q)\frac{y^m}{(q;q)_{m}}=(xy;q)_{\infty}{}_{2}\varphi_{1}\left(
        \begin{array}{c}
             q^n,0 \\
             -xy
        \end{array};
        q,by
        \right).
    \end{equation*}
\end{theorem}
\begin{proof}
    \begin{align*}
        \sum_{m=0}^{\infty}q^{\binom{m}{2}}\Phi_{m}^{(q^n)}(b,x|q)\frac{y^m}{(q;q)_{m}}&=\sum_{m=0}^{\infty}q^{\binom{m}{2}}\H_{n}(bD_{q})\{x^m\}\frac{y^m}{(q;q)_{m}}\\
        &=\H_{n}(bD_{q})\left\{\sum_{m=0}^{\infty}q^{\binom{m}{2}}\frac{(xy)^m}{(q;q)_{m}}\right\}\\
        &=(xy;q)_{\infty}{}_{2}\varphi_{1}\left(
        \begin{array}{c}
             q^n,0 \\
             xy
        \end{array};
        q, by
        \right).
    \end{align*}
\end{proof}

\begin{theorem}
    \begin{multline*}
        \sum_{m=0}^{\infty}\Phi_{m}^{(q^n)}(a,x|q)\Phi_{m}^{(q^k)}(b,y|q)z^m\\
        =\sum_{m=0}^{\infty}(q^k;q)_{m}(q;q)_{m}(bz)^m\sum_{i=0}^{\infty}\frac{(q^n;q)_{i}(yz)^i}{(q;q)_{i}^2(xyz;q)_{m+i+1}}\sum_{l=0}^{\infty}\frac{(q^{n+i};q)_{l}x^{m-l}}{(q^i;q)_{l}(q;q)_{l}(q;q)_{m-l}}.
    \end{multline*}
\end{theorem}
\begin{proof}
    \begin{align*}
        &\sum_{m=0}^{\infty}\Phi_{m}^{(q^n)}(a,x|q)\Phi_{m}^{(q^k)}(b,y|q)z^m\\
        &\hspace{1cm}=\sum_{m=0}^{\infty}\Phi_{m}^{(q^k)}(b,y|q)\H_{n}(aD_{q})\{(xz)^m\}\\
        &\hspace{1cm}=\H_{n}(aD_{q})\left\{\sum_{m=0}^{\infty}\Phi_{m}^{(q^k)}(b,y|q)(xz)^m\right\}\\
        &\hspace{1cm}=\H_{n}(aD_{q})\left\{\frac{1}{1-xyz}{}_{2}\varphi_{1}\left(
        \begin{array}{c}
             q^k,q \\
             qx yz
        \end{array};
        q, bxz 
        \right)\right\}\\
        &\hspace{1cm}=\sum_{l=0}^{\infty}\frac{(q^n;q)_{l}}{(q;q)_{l}}a^lD_{q}^l\left\{\sum_{m=0}^{\infty}\frac{(q^k;q)_{m}}{(xyz;q)_{m+1}}(b xz)^m\right\}.
    \end{align*}
\begin{align*}
    &\sum_{m=0}^{\infty}\Phi_{m}^{(q^n)}(a,x|q)\Phi_{m}^{(q^k)}(b,y|q)z^m\\
    &\hspace{1cm}=\sum_{m=0}^{\infty}(q^k;q)_{m}\sum_{l=0}^{\infty}\frac{(q^n;q)_{l}}{(q;q)_{l}}a^lD_{q}^l\left\{\frac{1}{(xyz;q)_{m+1}}(b xz)^m\right\}\\
    &\hspace{1cm}=\sum_{m=0}^{\infty}(q^k;q)_{m}\sum_{l=0}^{\infty}\frac{(q^n;q)_{l}}{(q;q)_{l}}a^l\sum_{i=0}^{l}\qbinom{l}{i}_{q}q^{i(i-l)}D_{q}^i\left\{\frac{1}{(xyz;q)_{m+1}}\right\}D_{q}^{l-i}(q^{i}b xz)^m\\
    &\hspace{1cm}=\sum_{m=0}^{\infty}(q^k;q)_{m}(bz)^m\sum_{l=0}^{\infty}\frac{(q^n;q)_{l}}{(q;q)_{l}}a^l\sum_{i=0}^{l}\frac{1}{(q;q)_{i}(q;q)_{l-i}}\frac{(yz)^i}{(xyz;q)_{m+i+1}}\frac{(q;q)_{m}}{(q;q)_{m-l+i}}x^{m-l+i}\\
        &\hspace{1cm}=\sum_{m=0}^{\infty}(q^k;q)_{m}(q;q)_{m}(bz)^m\sum_{i=0}^{\infty}\frac{(yz)^i}{(q;q)_{i}(xyz;q)_{m+i+1}}\sum_{l=i}^{\infty}\frac{(q^n;q)_{l}x^{m-l+i}}{(q;q)_{l}(q;q)_{l-i}(q;q)_{m-l+i}}\\
        &\hspace{1cm}=\sum_{m=0}^{\infty}(q^k;q)_{m}(q;q)_{m}(bz)^m\sum_{i=0}^{\infty}\frac{(q^n;q)_{i}(yz)^i}{(q;q)_{i}^2(xyz;q)_{m+i+1}}\sum_{l=0}^{\infty}\frac{(q^{n+i};q)_{l}x^{m-l}}{(q^i;q)_{l}(q;q)_{l}(q;q)_{m-l}}.
\end{align*}
\end{proof}

\begin{theorem}
    \begin{multline*}
        \sum_{m=0}^{\infty}\Phi_{m}^{(q^n)}(a,x|q)\Phi_{m}^{(q^k)}(b,y|q)\frac{z^m}{(q;q)_{m}}\\=\frac{(q^kbxz;q)_{\infty}}{(xyz,bxz;q)_{\infty}}\sum_{m=0}^{\infty}\frac{(q^n;q)_{m}}{(q;q)_{m}(q^kbxz;q)_{m}}(bz)^m\sum_{l=0}^{m}\qbinom{m}{l}_{q}\Phi_{l}^{(bxz)}(yz,bz)(q^kbz)^{m-l}.
    \end{multline*}
\end{theorem}
\begin{proof}
    \begin{align*}
        &\sum_{m=0}^{\infty}\Phi_{m}^{(q^n)}(a,x|q)\Phi_{m}^{(q^k)}(b,y|q)\frac{z^m}{(q;q)_{m}}\\
        &\hspace{1cm}=\sum_{m=0}^{\infty}\Phi_{m}^{(q^k)}(b,y|q)\H_{n}(aD_{q})\left\{\frac{(xz)^m}{(q;q)_{m}}\right\}\\
        &\hspace{1cm}=\H_{n}(aD_{q})\left\{\sum_{m=0}^{\infty}\Phi_{m}^{(q^k)}(b,y|q)\frac{(xz)^m}{(q;q)_{m}}\right\}\\
        &\hspace{1cm}=\H_{n}(aD_{q})\left\{\frac{(q^{k}bxz;q)_{\infty}}{(yxz,bxz;q)_{\infty}}\right\}\\
        &\hspace{1cm}=\frac{(q^kbxz;q)_{\infty}}{(xyz,bxz;q)_{\infty}}\sum_{m=0}^{\infty}\frac{(q^n;q)_{m}}{(q;q)_{m}(q^kbxz;q)_{m}}(bz)^m\sum_{l=0}^{m}\qbinom{m}{l}_{q}\Phi_{l}^{(bxz)}(yz,bz)(q^kbz)^{m-l}.
    \end{align*}
\end{proof}

\begin{theorem}
    \begin{equation*}
        \sum_{n=0}^{\infty}\sum_{m=0}^{\infty}\Phi_{n+m}^{(q^k)}(b,x|q)z^ny^m=\frac{1}{(1-xz)(1-xy)}\sum_{i=0}^{\infty}\frac{(q^k;q)_{i}(bz)^i}{(xz;q)_{i+1}}{}_{2}\varphi_{1}\left(
        \begin{array}{c}
             q^{k+i},q  \\
              q^{i+1}xy
        \end{array};
        q, by
        \right).
    \end{equation*}
\end{theorem}
\begin{proof}
From Eq.(\ref{eqn_leibniz}), we get
    \begin{align*}
        &\sum_{n=0}^{\infty}\sum_{m=0}^{\infty}\Phi_{n+m}^{(q^k)}(b,x|q)z^ny^m\\
        &\hspace{1cm}=\sum_{n=0}^{\infty}\sum_{m=0}^{\infty}\H_{k}(bD_{q})\{x^{n+m}\}z^ny^m\\
        &\hspace{1cm}=\H_{k}(bD_{q})\left\{\sum_{n=0}^{\infty}\sum_{m=0}^{\infty}(xz)^n(xy)^m\right\}\\
        &\hspace{1cm}=\H_{k}(bD_{q})\left\{\frac{1}{1-xz}\frac{1}{1-xy}\right\}\\
        &\hspace{1cm}=\sum_{l=0}^{\infty}\frac{(q^k;q)_{l}}{(q;q)_{l}}b^{l}D_{q}^{l}\left\{\frac{1}{1-xz}\frac{1}{1-xy}\right\}\\
        &\hspace{1cm}=\sum_{l=0}^{\infty}\frac{(q^k;q)_{l}}{(q;q)_{l}}b^{l}\sum_{i=0}^{l}\qbinom{i}{l}_{q}q^{i(i-l)}D_{q}^{i}\left\{\frac{1}{1-xz}\right\}D_{q}^{l-i}\left\{\frac{1}{1-q^{i}xy}\right\}\\
        &\hspace{1cm}=\sum_{l=0}^{\infty}(q^k;q)_{l}b^{l}\sum_{i=0}^{l}\frac{z^{i}}{(xz;q)_{i+1}}\frac{y^{l-i}}{(q^{i}xy;q)_{l-i+1}}\\
        &\hspace{1cm}=\sum_{i=0}^{\infty}\frac{(bz)^i}{(xz;q)_{i+1}}\sum_{l=0}^{\infty}\frac{(q^{k};q)_{l+i}}{(q^ixy;q)_{l+1}}(by)^{l}\\
        &\hspace{1cm}=\sum_{i=0}^{\infty}\frac{(q^k;q)_{i}(bz)^i}{(xz;q)_{i+1}}\sum_{l=0}^{\infty}\frac{(q^{k+i};q)_{l}}{(q^ixy;q)_{l+1}}(by)^{l}\\
        &\hspace{1cm}=\frac{1}{(1-xz)(1-xy)}\sum_{i=0}^{\infty}\frac{(q^k;q)_{i}(bz)^i}{(xz;q)_{i+1}}{}_{2}\varphi_{1}\left(
        \begin{array}{c}
             q^{k+i},q  \\
              q^{i+1}xy
        \end{array};
        q, by
        \right)
        \end{align*}
as claimed.        
\end{proof}

\begin{theorem}
    \begin{equation*}
        \sum_{n=0}^{\infty}\sum_{m=0}^{\infty}\Phi_{n+m}^{(q^k)}(b,x|q)\frac{y^n}{(q;q)_{n}}\frac{z^m}{(q;q)_{m}}=\frac{(q^{k}bz;q)_{\infty}}{(yx,xz,bz;q)_{\infty}}{}_{2}\varphi_{1}\left(
        \begin{array}{c}
             q^{k},xz  \\
             q^{k}bz
        \end{array};
        q, by
        \right).
    \end{equation*}
\end{theorem}
\begin{proof}
    \begin{align*}
        \sum_{n=0}^{\infty}\sum_{m=0}^{\infty}\Phi_{n+m}^{(q^k)}(b,x|q)\frac{y^n}{(q;q)_{n}}\frac{z^m}{(q;q)_{m}}
        &=\sum_{n=0}^{\infty}\sum_{m=0}^{\infty}\H_{k}(bD_{q})\{x^{n+m}\}\frac{y^n}{(q;q)_{n}}\frac{z^m}{(q;q)_{m}}\\
        &=\H_{k}(bD_{q})\left\{\sum_{n=0}^{\infty}\sum_{m=0}^{\infty}x^{n+m}\frac{y^n}{(q;q)_{n}}\frac{z^m}{(q;q)_{m}}\right\}\\
        &=\H_{k}(bD_{q})\left\{\sum_{n=0}^{\infty}\frac{(xy)^n}{(q;q)_{n}}\sum_{m=0}^{\infty}\frac{(xz)^m}{(q;q)_{m}}\right\}\\
        &=\H_{k}(bD_{q})\left\{\frac{1}{(xy,xz;q)_{\infty}}\right\}\\
        &=\frac{(q^{k}bz;q)_{\infty}}{(yx,xz,bz;q)_{\infty}}{}_{2}\varphi_{1}\left(
        \begin{array}{c}
             q^{k},xz  \\
             q^{k}bz
        \end{array};
        q, by
        \right).
    \end{align*}
\end{proof}

\end{document}